\documentclass[11pt,a4paper,oneside]{amsart}
\usepackage[scale=0.71]{geometry}
\usepackage[colorlinks=true,citecolor=blue]{hyperref}
\newtheorem{theorem}{Theorem}[section]

\usepackage{txfonts}

\theoremstyle{definition}
\newtheorem{remark}[theorem]{Remark}

\newcommand{\R}{\mathbb R}
\newcommand{\Ss}{\mathbb S}
\newcommand{\cal}{\mathcal}

\DeclareMathOperator{\sgn}{sgn}

\begin{document}

\title{Cutting convex curves}

\author{Andreas F. Holmsen}
\address{A.~F.~Holmsen\\ Department of Mathematical Sciences, KAIST\\ Daejeon, South Korea} 
\email{andreash@kaist.edu}

\author{J{\'a}nos Kincses}
\address{J{\'a}nos Kincses\\ University of Szeged, Bolyai Institute
Aradi v\'ertan\'uk tere 1., H-6720 Szeged, Hungary}
\email{kincses@math.u-szeged.hu}

\author{Edgardo Rold\'an-Pensado}
\address{E. Rold\'an-Pensado\\ Instituto de Matem\'aticas, UNAM campus Juriquilla\\ Quer\'etaro, M\'exico}
\email{e.roldan@im.unam.mx}


\thanks{The second author was supported by CONACyT project 166306.}

\begin{abstract}
We show that for any two convex curves $C_1$ and $C_2$ in $\R^d$ parametrized by $[0,1]$ with opposite orientations, there exists a hyperplane $H$ with the following property: For any $t\in [0,1]$ the points $C_1(t)$ and $C_2(t)$ are never in the same open halfspace bounded by $H$. This will be deduced from a more general result on equipartitions of ordered point sets by hyperplanes.
\end{abstract} 

\maketitle 

\section{Introduction} 

In \cite{HST2014} the following theorem is proved: {\em If $A_1$, $A_2$, $\ldots$, $A_n$ and $B_1$, $B_2$, $\ldots$, $B_n$ are the vertices of two convex polygons in the plane ordered cyclically with opposite orientation, then there exists a line that intersects each of the line segments $A_jB_j$. }

This result can be derived from a continuous version of the problem which has an elementary topological argument (which is what they do in \cite{HST2014}). The natural problem which is raised in \cite{HST2014} is to try to generalize this result to higher dimensions, and some partial results are proven for convex polytopes in $\R^3$ (but with some limitations). 

Here we will give a generalization of this theorem to arbitrary dimensions. Our proof is essentially different from the one given in \cite{HST2014} and uses notions from oriented matroid theory together with a basic fixed-point theorem. 

\medskip

A \emph{convex curve} in $\R^d$ is a continuous mapping $C \colon[0,1] \to \R^d$ which intersects every hyperplane at most $d$ times, meaning $\lvert \{t\in [0,1] : C(t) \in H \} \rvert\leq d$ for any hyperplane $H\subset \R^d$. A typical example of a convex curve in $\R^d$ is the so-called \emph{moment curve},
\[
 \left\{\left(t, t^2, \dots, t^d\right) : t\in [0,1]\right\},
\]
which has numerous applications in discrete and computational geometry. For instance, the convex hull of $n>d$ distinct points on the moment curve in $\R^d$ is a cyclic $d$-polytope \cite{Zie1995}, which is arguably the most useful example of a neighborly polytope. 

A convex curve is \emph{closed} if $C(0) = C(1)$, in which case we require that $\lvert \{ t\in [0,1) : C(t)\in H \} \rvert\leq d$ for any hyperplane $H\subset \R^d$. Notice that a closed convex curve in $\R^d$ exists only when the dimension $d$ is even. A typical example of a closed convex curve is the \emph{trigonometric moment curve},
\[
 \{(\cos (2\pi t), \sin (2\pi t), \cos (4\pi t), \sin (4\pi t), \dots, \cos (2d\pi t), \sin (2d\pi t)) : 0\leq t \leq 1\}.
\]

The convex hull of the trigonometric moment curve was first studied by Carath\'{e}odory \cite{Cara1907}, and its projections give rise to interesting examples of orbitopes and spectahedra \cite{SSS2011}. An important feature of a convex curve in $\R^d$ is the fact that for any $0\leq t_0 < t_1 < \dots < t_d\leq 1$, the determinant
\begin{equation}\label{det}
 \det \begin{bmatrix}
 C(t_0) & C(t_1) & \cdots & C(t_d) \\
 1 & 1 & \cdots & 1 
 \end{bmatrix}
\end{equation}
does not vanish, which is in fact a defining property of convex curves \cite{Sch1954}. (In the case of closed convex curves we naturally require that $t_d<1$.) This implies that the determinant \eqref{det} has the same sign for all choices $0\leq t_0 < t_1 < \dots < t_d\leq 1$, and therefore we may define the \emph{orientation} of a convex curve $C$ to be \emph{positive} or \emph{negative} according to the sign of the determinant \eqref{det}.

The main motivation behind this note is to report the following interesting property concerning \emph{pairs} of convex curves.

\begin{theorem}\label{convex curves}
 Let $C_1$ and $C_2$ be (closed) convex curves in $\R^{d}$ with opposite orientations. There exists a hyperplane $H$ such that the points $C_1(t)$ and $C_2(t)$ are never contained in the same open halfspace bounded by $H$.
\end{theorem}

For $d=2$ this is the main result shown in \cite{HST2014}. Somewhat surprisingly, the convexity plays a rather minor role. Theorem \ref{convex curves} will be deduced from a more general result concerning point sets, stated below as Theorem \ref{ordertypes}.

\section{Order-types}\label{sec:ot}

Let $A$ be a set of points in $\R^d$ which affinely span $\R^d$. The \emph{order-type} of $A$ is the set of signs of the determinants \begin{equation}\label{orientation}
 \det \begin{bmatrix}
 a_0 & a_1 & \cdots & a_d \\
 1 & 1 & \cdots & 1
 \end{bmatrix}
\end{equation}
indexed by the $(d+1)$-tuples $(a_0,a_1,\cdots,a_d)\in A^{d+1}$ with distinct entries. Notice that the condition that $A$ affinely spans $\R^d$ guarantees the existence of at least one $(d+1)$-tuple such that the determinant \eqref{orientation} is non-zero.
Usually, the notion of order-type is used with finite sets of points, however we will allow the possibility of $A$ being infinite.

The order-type defines an equivalence relation on sets of points in $\R^d$, in which two sets $A$ and $B$ are equivalent if there exists a bijection $\gamma:A\to B$ with
\begin{equation}\label{preserve}
 \sgn\det \begin{bmatrix}
 a_0 & a_1 & \cdots & a_d \\
 1 & 1 & \cdots & 1
 \end{bmatrix} = \sgn\det \begin{bmatrix}
 \gamma(a_0) & \gamma(a_1) & \cdots & \gamma(a_d) \\
 1 & 1 & \cdots & 1 \end{bmatrix}
\end{equation}
for all $(d+1)$-tuples $(a_0,a_1,\cdots,a_d)$ with distinct entries (see e.g. \cite{GP1983}).

To the other extreme, we say that the sets $A$ and $B$ have \emph{opposite} order-types if
\[
 \sgn\det \begin{bmatrix}
 a_0 & a_1 & \cdots & a_d \\
 1 & 1 & \cdots & 1
 \end{bmatrix} = -\sgn\det \begin{bmatrix}
 \gamma(a_0) & \gamma(a_1) & \cdots & \gamma(a_d) \\
 1 & 1 & \cdots & 1 \end{bmatrix}
\]
is satisfied instead of \eqref{preserve}. We say in this case that $\gamma$ is \emph{order-type reversing}.

\begin{theorem}\label{ordertypes}
 Let $A$ and $B$ be point sets in $\R^d$ which affinely span $\R^d$. If $\gamma:A\to B$ is an order-type reversing bijection, then there exists a hyperplane which intersects all the segments $ab$ with $b=\gamma(a)$.
\end{theorem}

\begin{remark}
 The condition on the affine span of the point sets could be weakened, but this would involve refining the notion of the order-type (since all the determinants \eqref{orientation} would vanish) and the statement of Theorem \ref{ordertypes} would become more technical. 
\end{remark}

\begin{remark}
 The order-type preserving or reversing property of a map includes that affine independent points are mapped to affine independent points. For ''nice'' infinite sets this property in itself is so strong that the map must be very special. For example, it is not hard to prove that if the set
 $A\subseteq {\mathbb R}^d$ $(d\ge3)$ contains the boundary points of a bounded open set then the map must be the restriction of a projective map for $A$.
\end{remark}

The proof of Theorem \ref{ordertypes} is given in the following section. To see how this theorem implies Theorem \ref{convex curves}, simply take $\gamma$ to be the function that maps $C_1(t)$ to $C_2(t)$ for every $t$.

\section{Proof of Theorem \ref{ordertypes}}\label{proof}
Here we prove a slightly more general statement given below in Theorem \ref{reversal} which will easily imply Theorem \ref{ordertypes}. It will be more convenient to reformulate this in linear terms as finite vector configurations in $\R^{d+1}$. Readers familiar with the theory of oriented matroids \cite{BVSWZ99} will recognize the concepts immediately.

Let $V = \{v_1, v_2, \dots, v_n\}$ be a finite configuration of non-zero vectors in $\R^{d+1}$, and assume that the linear span of $V$ is $(d+1)$-dimensional. Let $\Ss^d$ denote the unit sphere centered at the origin. For every $x\in \Ss^d$ we associate a \emph{sign vector}, $\sigma^{_V}(x) \in \{+,-,0\}^n$, by defining the $i$-th coordinate of $\sigma^{_V}(x)$ as
\[
 \sigma^{_V}(x)_i = \sgn{\langle x,v_i \rangle},
\]
where $\langle \cdot, \cdot \rangle$ denotes the usual Euclidean inner product. 

The set of points in $\Ss^d$ with the same sign vector forms an open topological cell, as it is the intersection of $\Ss^d$ with an open convex cone with apex at the origin. The set of all the cells forms a cell decomposition of $\Ss^d$, which we denote by ${\cal C}^{_V}$, and two such cell decompositions are called \emph{combinatorially equivalent} if their face posets are isomorphic. Notice that these cell decompositions are \emph{antipodal} in the sense that for a cell corresponding to a signed vector $\sigma^{_V}$ there is a (geometrically) antipodal cell which corresponds to the signed vector $-\sigma^{_V}$. 

A function $\gamma \colon V \to \R^{d+1}$ is \emph{orientation reversing} if
\[
 \sgn \det \left[ v_{i_0}, v_{i_1}, \dots, v_{i_{d}} \right] = - \sgn \det \left[ \gamma(v_{i_0}), \gamma(v_{i_1}), \dots, \gamma(v_{i_{d}}) \right]
\]
for all choices of indices $1\leq i_0 < i_1 < \dots < i_{d} \leq n$. Here $\left[ v_{i_0}, v_{i_1}, \dots, v_{i_{d}}\right]$ denotes the matrix in $\R^{(d+1)\times(d+1)}$ with the $v_{i_j}$ as column vectors. It is a well-known fact that the face poset of ${\cal C}^{_V}$ is determined by the set of signs of the determinants of the matrices $\left[ v_{i_0}, v_{i_1}, \dots, v_{i_{d}}\right]$. This corresponds to the equivalence between the covector axioms and chirotope axioms for oriented matroids (see e.g. \cite[Chapter 5]{BVSWZ99}). Moreover, it follows that if $\gamma$ is orientation reversing, then ${\cal C}^{_V}$ and ${\cal C}^{_{\gamma(V)}}$ are combinatorially equivalent.

\begin{theorem}\label{reversal}
 Let $V = \{v_1, \dots, v_n\}$ be a configuration of non-zero vectors in $\R^{d+1}$ which linearly spans $\R^{d+1}$. For any orientation reversing function $\gamma \colon V \to \R^{d+1}$ there exists a point $x\in \Ss^d$ such that the associated sign vectors satisfy $\sigma^{_V}(x) = - \sigma^{_{\gamma(V)}}(x)$.
\end{theorem}

\begin{proof}
The idea is to extend the function $\gamma$ to a homeomorphism $g \colon \Ss^d \to \Ss^d$ such that the point we are looking for is a fixed point of $g$. The existence of a fixed point is guaranteed by showing that the degree of $g$ equals $(-1)^d$, since any map from $\Ss^d$ to itself without fixed points is homotopic to the antipodal map and therefore has degree $(-1)^{d+1}$ (see e.g. \cite[Chapter 2.2]{Hat2002}).

\emph{Constructing the homeomorphism.}\\
First define $g$ from the vertices of ${\cal C}^{_V}$ (the $0$-cells) to the vertices of ${\cal C}^{_{\gamma(V)}}$ by mapping a vertex $x$ of ${\cal C}^{_V}$ with sign vector $\sigma^{_V}(x)$ to the unique vertex $y$ of ${\cal C}^{_{\gamma(V)}}$ with sign vector $\sigma^{_{\gamma(V)}}(y) = - \sigma^{_V}(x)$. Once $g$ has been defined from the $k$-skeleton of ${\cal C}^{_V}$ to the $k$-skeleton of ${\cal C}^{_{\gamma(V)}}$ we can extend the map continuously to the $(k+1)$-skeletons, since the boundary of each $(k+1)$-cell is homeomorphic to a $k$-sphere consisting of cells of dimension at most $k$. In this way, a cell of ${\cal C}^{_V}$ with sign vector $\sigma^{_V}$ is mapped by a homeomorphism to the unique cell of ${\cal C}^{_{\gamma(V)}}$ with sign vector $\sigma^{_{\gamma(V)}} = - \sigma^{_V}$.

\emph{Calculating the degree.}\\
It follows from a result of Shannon \cite[Lemma 1]{Sha1979} that the cell decomposition ${\cal C}^{_V}$ contains a simplicial $d$-cell, $\Delta$, and we denote the vertices of $\Delta$ by $x_0, x_1, \dots, x_d$. (This is where we use that the span of $V$ is $(d+1)$-dimensional.)
Let $X = \left[ x_0, x_1, \dots, x_d \right]$ and $Y = \left[ y_0, y_1, \dots, y_d \right]$, where the $y_i = g(x_i)$ denote the vertices of the simplicial $d$-cell $g(\Delta)$ of ${\cal C}^{_{\gamma(V)}}$. Since $g$ is a homeomorphism, its degree, $\deg(g)$, is either $+1$ or $-1$ and satisfies
\[
 \deg(g) = (\sgn\det X) \cdot (\sgn\det Y).
\]
Since $x_0, x_1, \dots, x_d$ are vertices of ${\cal C}^{_V}$ there exist vectors $v_{i_0}, v_{i_1}, \dots, v_{i_{d}} \in V$ such that
\[
 \left[ v_{i_0}, v_{i_1}, \dots, v_{i_d}\right]^T \cdot X
\]
is a diagonal matrix with non-zero entries $\alpha_0, \alpha_1, \dots, \alpha_d$ on its main diagonal. Similarly, we get that
\[
 \left[ \gamma(v_{i_0}), \gamma(v_{i_1}), \dots, \gamma(v_{i_d}) \right]^T \cdot Y
\]
is a diagonal matrix with non-zero entries $\beta_0, \beta_1, \dots, \beta_d$ on its main diagonal. Since $\gamma$ is orientation reversing, it follows that
\[
 \sgn \det \left[v_{i_0}, v_{i_1}, \dots, v_{i_d} \right] = - \sgn \det \left[ \gamma(v_{i_0}), \gamma(v_{i_1}), \dots, \gamma(v_{i_d}) \right],
\]
and by the definition of $g$ we get that $\alpha_i \cdot \beta_i <0$ for all $i = 0, \dots, d$. Therefore
\[
 (\sgn \det \left[ v_{i_0}, v_{i_1}, \dots, v_{i_d}\right]^T \cdot X) \cdot (\sgn\det \left[ \gamma(v_{i_0}), \gamma(v_{i_1}), \dots, \gamma(v_{i_d}) \right]^T \cdot Y) = (-1)^{d+1},
\]
which implies that $\deg(g) = (\sgn\det X) \cdot (\sgn\det Y) = (-1)^d$.
\end{proof}

\begin{proof}[Proof of Theorem \ref{ordertypes}]
First assume that $A=\{a_1,a_2,\dots,a_n\}$ and $B=\{b_1,b_2,\dots,b_n\}$ are finite, and $\gamma(a_i)=b_i$ for $i=1,\dots,n$. If we think of $\R^d$ as being embedded in $\R^{d+1}$ as the affine hyperplane $\{(x_1, x_2, \dots, x_d, 1): x_i \in \R\}$, the point sets $A$ and $B$ in $\R^d$ can be thought of as vector configurations $V = \{v_1, v_2, \dots, v_n\}$ and $W = \{w_1, w_2, \dots, w_n\}$, respectively. Since $A$ and $B$ affinely span $\R^d$, it follows that $V$ and $W$ linearly span $\R^{d+1}$, and the fact that $A$ and $B$ have opposite order-types means that $\gamma$ is orientation reversing in $\R^{d+1}$. By Theorem \ref{reversal} there is a vector $x\in \Ss^d \subset \R^{d+1}$ such that $\langle x, v_i\rangle = - \langle x, w_i \rangle$ for every $i = 1, 2, \dots, n$. Thus, the vectors $v_i$ and $w_i$ lie on opposite sides of the orthogonal complement $x^\perp$, and therefore $H = x^\perp \cap \{(x_1, x_2, \dots, x_d, 1) : x_i \in \R\}$ is a hyperplane which intersects each of the segments $a_ib_i$.

The infinite case follows by a simple approximation argument. For each $n>d$ define a point set $A_n\subset A$ with $n$ elements that affinely span $\R^d$, and let $B_n=\gamma(A_n)$. Then there is a hyperplane $H_n$ which intersects each of the segments $ab$ with $a\in A_n$ and $b=\gamma(a)$. The sequence of hyperplanes $\{H_n\}$ contains a subsequence which converges to a hyperplane $H$ with the desired properties.
\end{proof}

\section*{acknowledgements}
The research of the third author was supported by CONACyT project 166306.

\end{document}